\documentclass[12pt]{amsart}

\usepackage{tensor}     

\usepackage{tikz}
\usetikzlibrary{calc, decorations.pathreplacing, arrows.meta, patterns}

\usepackage[colorlinks=true,urlcolor=blue, citecolor=red,linkcolor=blue,linktocpage,pdfpagelabels, bookmarksnumbered,bookmarksopen]{hyperref}
\usepackage[hyperpageref]{backref}
\usepackage{cleveref}
\usepackage{xcolor}
\usepackage{amsthm} 
\usepackage{latexsym,amsmath,amssymb}
\usepackage{accents}
\usepackage[colorinlistoftodos,prependcaption,textsize=tiny]{todonotes}
\usepackage{a4wide}
\usepackage{soul}
\usepackage{mathtools} 
\usepackage{xparse} 
\usepackage{enumitem}		

\usepackage{calc}

\usepackage{accents}

\definecolor{indigo}{rgb}{0.29, 0.0, 0.51}
\definecolor{p1}{gray}{0.4}
\definecolor{p2}{gray}{0.6}
\definecolor{p3}{gray}{0.98}
\definecolor{p4}{gray}{0.8}
\definecolor{p5}{gray}{0.9}

plus 6pt minus 12pt
\newcommand{\sing}{\operatorname{sing}}

\def\eps{\varepsilon}

\def\id{{\rm id\, }}

\def\n{{\mathcal N}}

\def\n{{\mathcal{M}}}
\def\n{{\mathcal N}}

\def\S{{\mathbb S}}

\renewcommand{\div}{{\rm div}}

\newcommand{\pheq}{\phantom{=}}

\newtheorem{theorem}{Theorem}
\newtheorem{lemma}[theorem]{Lemma}

\newtheorem{proposition}[theorem]{Proposition}

\newtheorem{definition}[theorem]{Definition}

\def\supp{{\rm supp\,}}

\newcommand{\dif}{\,\mathrm{d}}
\newcommand{\dx}{\dif x}

\newcommand{\R}{\mathbb{R}}

\newcommand{\cN}{\mathcal{N}}

\newcommand{\brac}[1]{\left (#1 \right )}
\newcommand{\abs}[1]{\left |#1 \right |}

\newcommand{\la}{\mathopen{}\mathclose\bgroup\left\langle}
\newcommand{\ra}{\aftergroup\egroup\right\rangle}

\newcommand{\barint}{
\rule[.036in]{.12in}{.009in}\kern-.16in \displaystyle\int }

\newcommand{\barcal}{\mbox{$ \rule[.036in]{.11in}{.007in}\kern-.128in\int $}}

\def\mvint_#1{\mathchoice
          {\mathop{\vrule width 6pt height 3 pt depth -2.5pt
                  \kern -8pt \intop}\nolimits_{\kern -3pt #1}}%
          {\mathop{\vrule width 5pt height 3 pt depth -2.6pt
                  \kern -6pt \intop}\nolimits_{#1}}%
          {\mathop{\vrule width 5pt height 3 pt depth -2.6pt
                  \kern -6pt \intop}\nolimits_{#1}}%
          {\mathop{\vrule width 5pt height 3 pt depth -2.6pt
                  \kern -6pt \intop}\nolimits_{#1}}}

\numberwithin{theorem}{section} \numberwithin{equation}{section}

\def\XXint#1#2#3{{\setbox0=\hbox{$#1{#2#3}{\int}$}
     \vcenter{\hbox{$#2#3$}}\kern-.5\wd0}}

\let\latexchi\chi
\makeatletter
\renewcommand\chi{\@ifnextchar_\sub@chi\latexchi}
\newcommand{\sub@chi}[2]{
  \@ifnextchar^{\subsup@chi{#2}}{\latexchi^{}_{#2}}%
}
\newcommand{\subsup@chi}[3]{
  \latexchi_{#1}^{#3}%
}
\makeatother

\usepackage{mathtools} 

\DeclareMathOperator{\divg}{div}

\title[Regularity of minimizing $p$-harmonic maps from $B^3$ to $\S^3$]{Regularity of minimizing $p$-harmonic maps from $B^3$ to $\S^3$}

\date{\today}
\author{Katarzyna Mazowiecka}
\address[Katarzyna Mazowiecka]{
Institute of Mathematics, %
University of Warsaw,
Banacha 2,
02-097 Warszawa, Poland}
\email{k.mazowiecka@mimuw.edu.pl}

\author{Micha\l{} Mi\'{s}kiewicz}
\address[Micha\l{} Mi\'{s}kiewicz]{
Institute of Mathematics, %
University of Warsaw,
Banacha 2,
02-097 Warszawa, Poland}
\email{m.miskiewicz@mimuw.edu.pl}

 \author{Patryk Tokarczuk}
\address[Patryk Tokarczuk]{ %
University of Warsaw,
Banacha 2,
02-097 Warszawa, Poland}
\email{p.tokarczuk@student.uw.edu.pl}

\begin{document}

 \begin{abstract}
We prove that every minimizing $p$-harmonic map from $B^3$ into $\mathbb S^3$ is locally $C^{1,\alpha}$ for some $\alpha\in(0,1)$, for every $p>p_0$, where $p_0= \frac{7-\sqrt{17}}{2}\approx 1.44$. This closes the gap between the previously established regularity ranges $[2,2.642]\cup[2.961,3]$ and, in particular, yields full interior regularity for all $p\ge2$. The result also extends regularity to the subquadratic range $p_0<p<2$.
 \end{abstract}

\sloppy

\subjclass[2010]{58E20, 35B65, 35J60}
\maketitle
\tableofcontents
\sloppy

\section{Introduction}
Let $\Omega \subset \R^n$ be an open domain, let $\mathcal N \subset \R^L$ be a smooth closed Riemannian manifold, and let $1<p<\infty$. We say that $u\in W^{1,p}(\Omega,\mathcal N)$ is a minimizing $p$-harmonic map\footnote{In the case when $p=2$ we simply say harmonic.} if it minimizes the $p$-Dirichlet energy 
\begin{equation}\label{eq:p-energy}
 E_p(u) \coloneqq \int_{\Omega } |\nabla u|^p \dx 
\end{equation}
among all maps $v\in W^{1,p}(\Omega, \n)$ with $u=v$ on $\partial\Omega$ in the trace sense. Here, 
\[
 W^{1,p} (\Omega, \n) \coloneqq \{v\in W^{1,p}(\Omega, \R^L)\colon v(x) \in \n \text{ for a.e. } x\in \Omega\}.
\]
It is well known that such maps can be singular. In fact, the appearance of singularities may be forced by the topology: Consider, for example, the class of continuous maps $C^0_{\id}(B^3,\S^2)$ with $\id\colon \S^2\to\S^2$ serving as the boundary condition. By Hopf lemma we know that this class is empty. However, $\frac{x}{|x|}\in W^{1,2}_{\id}(B^3,\S^2)\coloneqq \{v\in W^{1,2}(B^3,\S^2)\colon v\big\rvert_{\partial\Omega}=\id \text{ in the trace sense}\}$. Hence, any minimizing harmonic map corresponding to the boundary data $\id\colon \S^2 \to \S^2$ needs to be discontinuous. 

As shown by Schoen and Uhlenbeck for $p=2$ in their celebrated paper \cite{SU1} and then followed by Hardt and Lin \cite{HLp} for general $p>1$, minimizers enjoy \emph{partial regularity}, that is
\begin{equation}\label{eq:singularsetestimate}
 \dim_{\mathcal H} \sing u \le n - \lfloor p \rfloor -1,
\end{equation}
where $\sing u$ is the singular set of $u$ and $\dim_{\mathcal{H}}$ is the Hausdorff dimension.

This estimate is optimal for general target manifolds. Indeed, consider the map $\frac{x'}{|x'|}\colon B^{\lfloor p \rfloor +1}\to\S^{\lfloor p \rfloor}$, which is minimizing $p$-harmonic \cite{BCL, Lin-remark,CG,HLW}, and extend it to $B^n$ by
\[
(x',x'')\mapsto \frac{x'}{|x'|}\quad \text{ for } (x',x'')\in\R^n = \R^{\lfloor p \rfloor +1}\times \R^{n-\lfloor p \rfloor -1}
\]
independently of $x''$. 

However, the appearance of singularities is not just purely topological --- singularities may be energy efficient. Let us return to the example of the identity boundary map. This time we remove the topological obstacle by enlarging the target manifold to $\S^n$. Consider $\varphi\colon \S^{n-1}\to \S^n$ given by $\varphi(x)=(x,0)$, then  both $W^{1,p}_{\varphi}(B^n,\S^{n})\neq \emptyset$ and $C^0_{\varphi}(B^n,\S^n)\neq \emptyset$. The equatorial map $u^* \in W^{1,p}_{\varphi}(B^n,\S^{n})$ for $p<n$
\[
u^*\colon B^n\to\S^n,\quad u^*(x)\coloneqq \left(\frac{x}{|x|},0\right)
\]
turns out to be $p$-minimizing in large dimensions. More precisely for $p=2$ J\"{a}ger--Kaul \cite{Jager-Kaul} proved that $u^*$ is minimizing harmonic if and only if $n\ge 7$. This was generalized by Hong \cite{Hong2000}\footnote{The result can be in fact generalized to cover the case $1<p<2$.}: for $2\le p<n$ the equatorial map $u^*$ is $p$-minimizing if and only if $2\leq p<n-\sqrt{n-1}$.

For special target manifolds, the bound \eqref{eq:singularsetestimate} on the dimension of the singular set can be improved. This line of research was initiated by Schoen--Uhlenbeck \cite{SU3}. Consider, in particular, minimizing harmonic maps from $B^n$ into $\S^n$. As noted above, such maps may have singularities when $n\ge7$. In contrast, they are regular for $n=1$ by the Morrey embedding theorem, for $n=2$ by the classical result of Morrey \cite{Morrey}, and for $n=3$ by \cite{SU3}. Finally, Schoen--Uhlenbeck's method was ingeniously improved by Okayasu \cite{Okayasu94} and, recently, Li \cite{Li2026} which results in regularity in the remaining cases $n=4,5,6$.

Nakauchi \cite{Nakauchi01} and Xin--Yang \cite{XinYang} extended these methods to minimizing $p$-harmonic maps into spheres and, more generally, into manifolds with positive sectional curvature. Their results, however, do not yield optimal regularity conclusions. In particular, they do not establish regularity for minimizing $p$-harmonic maps from $B^3$ into $\S^3$ beyond the harmonic case $p=2$.

In \cite{Gastel19}, Gastel developed a connection between the theory of Cosserat
micropolar elasticity and $p$-harmonic maps. He showed that singular minimizers
in the Cosserat model lead to singular minimizing $p$-harmonic maps from $B^3$ to $\S^3$.
Therefore, to obtain regularity for Cosserat minimizers, it suffices to rule out the existence of
singular minimizing $p$-harmonic maps from $B^3$ to $\S^3$.
Gastel established regularity of $p$-minimizing maps from $B^3$ to $\S^3$ for $2\leq p\leq\frac{32}{15}\approx 2.133$.
Later, in \cite{MM}, Mazowiecka and Miśkiewicz improved this result and established regularity for $p\in[2,2.366]\cup[2.961,3)$, which was then further improved by Gastel, Mazowiecka, Miśkiewicz in \cite{GMM}, where regularity is obtained for the range $p\in[2,2.642]$.
These results leave open the gap $(2.642,2.961)$, where regularity of minimizing $p$-harmonic maps from $B^3$ to $\S^3$
is not known.

Here we fill this gap and prove that all minimizing $p$-harmonic maps from $B^3$ to $\S^3$ are regular for all $p\in(2,3)$.
Furthermore, we also prove regularity in the subquadratic case for $p\in[p_0,2]$, where $p_0\coloneqq\tfrac{7 - \sqrt{17}}{2} \approx 1.438$.
In summary, we obtain:

\begin{theorem}\label{regularity intro}
Minimizing $p$-harmonic maps $u\in W^{1,p}(B^3,\S^3)$ are regular for $p\in[p_0,3)$, where $p_0\coloneqq\tfrac{7 - \sqrt{17}}{2} \approx 1.438$.
\end{theorem}

Combined with the results of Hardt--Lin \cite{HLp} (regularity in the case $p=n=3$), Theorem \ref{regularity intro} establishes regularity of $p$-harmonic maps from $B^3$ to $\S^3$
for all $p\geq p_0$.

\subsection*{Acknowledgements}
The project is co-financed by:
\begin{itemize}
 \item the Polish National Agency for Academic Exchange within Polish Returns Programme -- BPN/PPO/2021/1/00019/U/00001 (K.M., P.T.)
 \item the National Science Centre grant 2022/01/1/ST1/00021 (K.M., P.T.)
 \item the National Science Centre grant 2025/59/D/ST1/03316 (M.M.)
\end{itemize}

\textbf{Usage of LLM:} The authors declare that no generative artificial intelligence (AI) tools, large language models, or automated proof systems were used in the conception, derivation, proof, or any other mathematical content of this work.

\section{Outline of the Proof}\label{outline} The regularity theory of Scheon--Uhlenbeck of minimizing maps relies on the study of minimizing tangent maps.

\begin{definition}
    A $p$-harmonic map $u\colon \R^n\to \n$ is called a \textit{tangent map} if it is $0$-homogeneous.
    If additionally $u$ minimizes the $p$-energy $E_p$ on compact subset of $\R^n$, we call $u$ a
    \textit{minimizing $p$-harmonic tangent map}.
\end{definition}

The following theorem makes this connection precise.

\begin{theorem}[{\cite[Theorem~IV]{SU1}, \cite[Theorem 4.5.]{HLp}}]\label{hl thm}
Suppose $\ell$ is the largest integer such that any minimizing $p$-harmonic tangent map from the unit ball in $\R^j$
into $\S^k$ is a constant map for each $j=1,\dots,\ell$.
Then the interior singular set of any $p$-minimizer $u\in W^{1,p}(B^n,\S^k)$ is empty in case $n<\ell+1$, is a discrete set in case $n=1+\ell$,
and has Hausdorff dimension at most $n-\ell-1$ in case $n\geq \ell+1$. Moreover, $\ell\geq [p]$.
\end{theorem}


The proof of Theorem \ref{regularity intro} follows a similar strategy as \cite{GMM}, which in turn follows
the arguments of Schoen and Uhlenbeck \cite{SU3}. In view of Theorem \ref{hl thm}, it suffices to study minimizing $p$-harmonic tangent maps
$u\colon \S^2\to\S^3$ (since such maps are $0$-homogeneous, it suffices to consider their restriction to the sphere $\S^2$).
Once we show that all such maps are constant, regularity follows. 

Therefore, let $u\colon \S^2\to\S^3$ be a minimizing $p$-harmonic tangent map. As a corollary of Theorem \ref{hl thm}, we obtain that
$u$ is of class $C^1$, see \Cref{C1}. This allows us to apply the results of Duzaar--Fuchs \cite{DuzaarFuchs} and Gastel--Grotowski--Kronz \cite{gasteletal}
(see Lemma \ref{second derivatives} below),
which provide information about the second derivatives of $u$.
By perturbing $u$, computing the second variation of the $p$-energy, and exploiting the stability condition, we obtain the inequality
\begin{equation}\label{stab}
\begin{split}
 (1+\gamma)^2&\int_{\S^2} |\nabla u|^{p-2+2\gamma}\big\{ 3|\nabla|\nabla u||^2+(p-2)\abs{\Delta_\infty u}^2\big\}\\
 & \quad\ge (3-p)\int_{\S^2} |\nabla u|^{p+2+2\gamma}-\frac 34 (3-p)^2\int_{\S^2}|\nabla u|^{p+2\gamma},
\end{split}
\end{equation}
where $\gamma$ is a nonpositive parameter over which we will optimize, and
\begin{equation} \label{eq:infinity-laplacian}   
\Delta_\infty u\coloneqq \left\langle \frac{\nabla u}{|\nabla u|}, \nabla |\nabla u|\right\rangle 
= \nabla^2 u \left( \frac{\nabla u}{|\nabla u|}, \frac{\nabla u}{|\nabla u|} \right)
\end{equation}
is the normalized $\infty$-Laplacian of $u$. 
Then, using the Bochner--Weitzenb\"ock formula together with the $p$-harmonic map equation for $u$, we derive
 \begin{equation}\label{boch}
 \begin{split}
   \int_{\S^2} |\nabla u|^{p-2+2\gamma}&\bigg\{|\nabla^2 u|^2+(p-2+2\gamma)|\nabla |\nabla u||^2+2\gamma(p-2)\abs{\Delta_\infty u}^2 \bigg\}  \\
  &\le \frac12 \int_{\S^2} |\nabla u|^{p+2+2\gamma} - \int_{\S^2}|\nabla u|^{p+2\gamma}.
 \end{split}
\end{equation}
Our goal is to combine these two inequalities as efficiently as possible. To this end, we derive the mixed Kato--Cauchy--Schwarz inequality
\begin{equation}\label{KCS}
    \begin{split}
        &3|\nabla |\nabla u||^2 + (p-2) \abs{\Delta_\infty u}^2 \\
        &\le C_p \left( |\nabla^2 u|^2+(p-2+2\gamma)|\nabla |\nabla u||^2+2\gamma(p-2)\abs{\Delta_\infty u}^2\right)
    \end{split}
\end{equation}
 with the constant
 \[
 C_p \coloneqq \begin{dcases}
 \frac{p+1}{(p + 2\gamma)(p-1)}\quad&\text{if $1<p<2$,}
 \\
\frac{3}{p+2\gamma}&\text{if $2\leq p<3$.}
\end{dcases}
 \]
The inequality \eqref{KCS} gives a pointwise bound between the integrands in \eqref{stab} and \eqref{boch}.
By combining these three inequalities, we obtain
 \begin{equation}\label{A,B estimate intro}
     A(p,\gamma)\int_{\S^2}\abs{\nabla u}^{p+2+2\gamma}+B(p,\gamma)\int_{\S^2}\abs{\nabla u}^{p+2\gamma}\leq 0
 \end{equation}
for some numbers $A(p,\gamma)$ and $B(p,\gamma)$. We then wish to identify the range of $p$ for which there
exists an admissible parameter $\gamma$ such that both $A(p,\gamma)$ and $B(p,\gamma)$ are nonnegative, and at least one of them is positive.
For such a pair $(p,\gamma)$, the estimate \eqref{A,B estimate intro} forces $u$ to be constant,
which in turn implies that all minimizing $p$-harmonic maps from $B^3$ to $\S^3$ are regular. We show that such $\gamma$
can be found, provided
that $p\geq p_0$, where $p_0\coloneqq\tfrac{7 - \sqrt{17}}{2} \approx 1.438$.
Even though the inequalities \eqref{stab} and \eqref{boch} are the natural starting points, proving them rigorously requires reversing
the order. Therefore, we prove \eqref{KCS} first, and only then \eqref{boch} and \eqref{stab}.

\underline{There is one major obstacle in completing this approach}: insufficient regularity of the quantities involved. First, if we choose $\gamma$ optimally, the term $|\nabla u|^{p-2+2\gamma}$ in both \eqref{stab} and \eqref{boch} may have a negative exponent. For this reason, we introduce an additional regularization parameter $\delta \to 0$ and carefully consider the limiting behavior as $\delta \to 0$.

Second, $p$-minimizers in general do not have second derivatives that we exploit above.
However, once we know that a~$p$-minimizer is of class $C^1$,
the following results apply:

\begin{lemma}\label{second derivatives}
    Let $\Omega\subset\R^n$ be an open domain and let $\mathcal{N}\subset\R^{k+1}$ be a Riemannian manifold.
    Let $p>1$ and $u\in W^{1,p}(\Omega,\mathcal{N})$ be a $C^1$ weakly $p$-harmonic map. 
    \begin{enumerate}[label=$(\arabic*)$]
        \item\label{it:DuzaarFuchs} If $p\geq 2$, then $|\nabla u|^{(p-2)/2}\nabla u\in W^{1,2}_\textnormal{loc}(\Omega)$. \textnormal{\cite[Lemma 2.2.]{DuzaarFuchs}}    
        \item\label{it:Gasteletal} If $1<p<2$, then $u\in W^{2,p}_\textnormal{loc}(\Omega,\cN)$ and $|\nabla u|^{p-2}|\nabla^2 u|^2\in L^1_\textnormal{loc}(\Omega)$. \textnormal{\cite[Lemma~2]{gasteletal}} 
    \end{enumerate}
    Moreover, for $1<p<3$ we have $u\in W^{2,2}_{loc}(\Omega)$.
\end{lemma}

\begin{proof}
If $1<p<2$, the local boundedness of $|\nabla u|$, together
with the weighted estimate in \ref{it:Gasteletal}, implies that
$u\in W^{2,2}_{\mathrm{loc}}(\Omega,\mathcal N)$.

It remains to consider $2<p<3$.
By \cite[Theorem~3.1]{HLp},
$u\in C^{1,\alpha}_{\mathrm{loc}}(\Omega,\mathcal N)$
for some $\alpha\in(0,1)$.
Let us write the Euler--Lagrange equation as
\[
\div\brac{|\nabla u|^{p-2}\nabla u}=f.
\]
The right-hand side is a quadratic expression in the components of $|\nabla u|^{(p-2)/2}\nabla u$,
with coefficients determined by the second fundamental form
of $\mathcal N$ evaluated at $u$.
These coefficients are locally $C^1$.
Thus, assertion \ref{it:DuzaarFuchs} give
$f\in W^{1,1}_{\mathrm{loc}}(\Omega)$.
Moreover, the $C^{1,\alpha}_{\mathrm{loc}}$ regularity of $u$
implies that $f\in C^{0,\beta}_{\mathrm{loc}}(\Omega)$
for some $\beta>0$. We may therefore apply
\cite[Theorem~2.2(ii)]{MR4825227}
to conclude that
$u\in W^{2,2}_{\mathrm{loc}}(\Omega,\mathcal N)$;
see also \cite{Cellina,Miskiewicz}.
\end{proof}

Furthermore, if a minimizing $p$-harmonic map $u\colon \Omega\to \S^k$ is $C^1$, then, by the regularity theory of Hardt and Lin \cite{HLp},
it is automatically $C^{1,\alpha}$ for some $\alpha\in(0,1)$. 
On each of the sets $\Omega_\eps = \{ |\nabla u| > \eps \}$ the $p$-harmonic map equation is uniformly elliptic, and can use the standard elliptic theory to bootstrap $C^{1,\alpha}$ to $C^\infty$. 
In summary, we have the following Lemma:

\begin{lemma}\label{smooth u}
    Let $u\in W^{1,p}(\Omega,\S^k)$ be a $C^1$ minimizing $p$-harmonic map. Then $u$ is smooth on the open set $\{|\nabla u|>0\}$.
\end{lemma}
%
\section{Mixed Kato--Cauchy--Schwarz Inequality}

In what follows, we consider $p>1.3$. The reason is that numerical experiments show that our method yields regularity only for $p\geq p_0\approx 1.44$,
and the assumption $p>1.3$ slightly simplifies the proof. It is possible to extend the theorems below to $p>1$, but the range of admissible
parameters $\gamma$ would be more complicated.

\begin{lemma}[Mixed Kato--Cauchy--Schwarz Inequality]\label{le:mixedKCSwithgamma}
Let $u \colon \S^2\to \S^3$ be a $C^1$ minimizing $p$-harmonic tangent map for some $1.3<p<3$, and let $\gamma\in (-\frac{1}{2},0]$ if $p\in(1.3,2)$ and
$\gamma\in (-1,0]$ if $p\in[2,3)$.
Then at all points where $\nabla u \neq 0$ we have the estimate
\begin{equation}\label{eq:mixedKCSwithgamma}
    \begin{split}
        &3|\nabla |\nabla u||^2 + (p-2) \abs{\Delta_\infty u}^2 \\
        &\le C_p \left( |\nabla^2 u|^2+(p-2+2\gamma)|\nabla |\nabla u||^2+2\gamma(p-2)\abs{\Delta_\infty u}^2\right),
    \end{split}
\end{equation}
 where 
 \[
 C_p \coloneqq \begin{dcases}
 \frac{p+1}{(p + 2\gamma)(p-1)}\quad&\text{if $1<p<2$,}
 \\
\frac{3}{p+2\gamma}&\text{if $2\leq p<3$.}
\end{dcases}
 \]
\end{lemma}
\begin{proof}
    We follow the proof of \cite[Lemma 5.2 (c)]{MM}. Let $q\in \S^2$ be such that $\nabla u(q)\neq 0$.
    By Lemma \ref{smooth u}, $u$ is smooth ina neighborhood of $q$.
    Since \eqref{eq:mixedKCSwithgamma} is a pointwise inequality, we shall work at this one particular point $q$.
    First, choose exponential coordinates around $q\in\S^2$ and $u(q)\in \S^3$. Then, the $p$-harmonic map equation at the point $q$
    simplifies to
    \begin{equation}\label{exp pharmonic}
        0=  \partial_{ii}u^\alpha+(p-2) \partial_{i}u^\alpha \partial_{j}u^\beta \partial_{ij}u^\beta\quad \text{for $\alpha=1,2,3$},
    \end{equation}
    where $\alpha,\beta$ denote the coordinates in the target space, $i,j$ denote the coordinates in the base space, and we sum over repeated indices. For the sake of simplicity, we may assume $|\nabla u(q)|=1$, as otherwise we would work with the vector $\frac{\nabla u(q)}{|\nabla u(q)|}$.
    Note that we still have freedom to choose orthonormal bases for the tangent spaces $T_q\S^2$ and $T_{u(q)}\S^3$.  
    First, choose an orthonormal basis $e_1,e_2$ for the tangent space $T_q\S^2$ for which the vectors $\partial_1u=\nabla u(q)\cdot e_1$
    and $\partial_2 u=\nabla u(q)\cdot e_2$ in $T_{u(q)}\S^3$ are orthogonal. Then, choose an orthonormal basis $E_1,E_2,E_3$ for $T_{u(q)}\S^3$ for
    which $\partial_1 u=a_1E_1$ and $\partial_2 u=a_2 E_2$ for some $a_1,a_2\geq0$. With respect to these bases,
    \begin{equation}
        \nabla u(q)=\begin{pmatrix}
            a_1 & 0
            \\
            0 & a_2
            \\
            0  & 0 
        \end{pmatrix},
    \end{equation}
    and the system \eqref{exp pharmonic} takes the form
    \begin{equation}\label{simplified system at q}
        \begin{cases}
            0= \partial_{11} u^1 +\partial_{22} u^1 + (p-2)a_1(a_1\partial_{11}u^1+a_2\partial_{12} u^2)
            \\
            0= \partial_{22} u^2 +\partial_{11} u^2 + (p-2)a_2(a_2\partial_{22}u^2+a_1\partial_{12} u^1)
            \\
            0= \partial_{11} u^3 +\partial_{22} u^3.
        \end{cases}
    \end{equation}
    The terms appearing in \eqref{eq:mixedKCSwithgamma} (evaluated at the point $q$) can be written as 
    \begin{equation}
        \begin{split}
            |\nabla u|^2(q)&=1=a_1^2+a_2^2,
            \\
            |\nabla|\nabla u||^2(q)&= (a_1 \partial_{11}u^1+ a_2 \partial_{12}u^2)^2+(a_2 \partial_{22}u^2+ a_1 \partial_{12}u^1)^2,
            \\
            \abs{\Delta_\infty u}^2(q)&= a_1^2(a_1 \partial_{11}u^1+ a_2 \partial_{12}u^2)^2+a_2^2(a_2 \partial_{22}u^2+ a_1 \partial_{12}u^1)^2,
            \\
            |\nabla^2 u|^2(q)&=\left( (\partial_{11}u^1)^2+2(\partial_{12}u^2)^2+(\partial_{22}u^1)^2 \right)
            \\
            &\quad +\left( (\partial_{22}u^2)^2+2(\partial_{12}u^1)^2+(\partial_{11}u^2)^2 \right)
            \\
            &\quad +\left( (\partial_{11}u^3)^2+2(\partial_{12}u^3)^2+(\partial_{22}u^3)^2 \right).
        \end{split}
    \end{equation}
By exploiting the first two equations of the system \eqref{simplified system at q}, it suffices to show that
\begin{equation}\label{x,y reduction}
\begin{split}
&3(\theta_1 x + \theta_2 y)^2 + (p-2) \theta_1^2 (\theta_1 x + \theta_2 y)^2 \\
&\le C \Big( x^2 + 2y^2 + (x + (p-2) \theta_1 (\theta_1 x + \theta_2 y))^2 
\\
&\quad + (p-2+2\gamma)(\theta_1 x + \theta_2 y)^2 +2\gamma(p-2)\theta_1^2(\theta_1x+\theta_2y)^2\Big)
\end{split}
\end{equation}
for all $x,y\in\R$ and $\theta_1,\theta_2\ge0$ satisfying $\theta_1^2+\theta_2^2 = 1$.
Indeed, by first taking $x=\partial_{11}u^1,y=\partial_{12}u^2,\theta_1=a_1,\theta_2=a_2$, we see from \eqref{simplified system at q} that
\[
-\partial_{22}u^1=x+(p-2)\theta_1(\theta_1x+\theta_2 y).
\]
Applying \eqref{x,y reduction} yields part of \eqref{eq:mixedKCSwithgamma}. The other part follows by taking
$x=\partial_{22}u^2,y=\partial_{12}u^1,\theta_1=a_2,\theta_2=a_1$. In this way, we obtain a slightly stronger inequality, without the terms in $ |\nabla^2 u|^2(q)$
involving $u^3$. This is not a problem, since the full Hessian $|\nabla^2 u|^2$ appears only on the right hand side of \eqref{eq:mixedKCSwithgamma}.

We introduce the rotated coordinates $z = \theta_1 x + \theta_2 y$, $w = \theta_2 x - \theta_1 y$. Notice that $x^2+y^2 = z^2+w^2$, $x = \theta_1 z + \theta_2 w$, and thus we can rearrange the above to the form
\[
\begin{split}
&3z^2 + (p-2) \theta_1^2 z^2\\
&\le C \left( (2+\theta_1^2p(p-2))z^2 +2w^2+2(p-2)\theta_1\theta_2zw+(p-2+2\gamma)z^2+2\gamma(p-2)\theta_1^2z^2\right),
\end{split}
\]
which can be further simplified to 
\[
Az^2+2Bzw+(2C)w^2\ge 0,
\]
where
\[
A\coloneqq C(p+2\gamma)(1+\theta_1^2(p-2))-3-\theta_1^2(p-2)
\]
and
\[
B\coloneqq C(p-2)\theta_1\theta_2.
\]
Thus \eqref{eq:mixedKCSwithgamma} holds if $C>0$ is such that the quadratic form
\[
(z,w)\mapsto Az^2+2Bzw+(2C)w^2
\]
is positive definite. Since $2C>0$, this form is positive definite if and only if its determinant is nonnegative, i.e., 
\begin{align*}
    0&\leq A\cdot (2C)-B^2
    \\
    &=C^2\left(2(p+2\gamma)(1+\theta_1^2(p-2))-
    \theta_1^2\theta_2^2(p-2)^2\right)
    \\
    &\quad -C\left(2(3+\theta_1^2(p-2))\right).
\end{align*}
Note that for $1.3<p<2$ and $-\frac{1}{2}<\gamma\leq0$ we have
\begin{align*}
  2(p+2\gamma)(1+\theta_1^2(p-2))-
    \theta_1^2\theta_2^2(p-2)^2&\geq 2(p+2\gamma)(1+(p-2))-\frac{1}{4}(p-2)^2
    \\
    &\geq 2(1.3-1)(1.3-1)-\frac{1}{4}(1.3-2)^2
    \\
    &=0.0575
    \\
    &>0.
\end{align*}
This is the reason we assume $p>1.3$.
Similarly, for $2\leq p<3$ and $-1<\gamma\leq 0$, we have
\begin{align*}
  2(p+2\gamma)(1+\theta_1^2(p-2))-
    \theta_1^2\theta_2^2(p-2)^2&\geq 2(p-2)-\frac{1}{4}(p-2)^2
    \\
    &=\frac{1}{4}(p-2)(10-p)
    \\
    &>0.
\end{align*}
Therefore the condition on $C$ is
\[
C\geq \frac{2(3+\theta_1^2(p-2))}{2(p+2\gamma)(1+\theta_1^2(p-2))-
    \theta_1^2\theta_2^2(p-2)^2}
\]
for all $\theta_1, \theta_2\geq0$
with $\theta_1^2+\theta_2^2=1$.
In the case $1.3<p<2$, this is equivalent to 
\[
C\geq \frac{p+1}{(p-1)(p+2\gamma)},
\]
while for $2\leq p<3$ we have 
\[
C\geq \frac{3}{p+2\gamma}.
\] 
\end{proof}

\section{Bochner Inequality}

Now that we have established Lemma \ref{le:mixedKCSwithgamma}, we can prove the inequality \eqref{boch}.

\begin{theorem}[Improved Bochner Inequality]\label{thm:Bochnerwithgamma}
    Let $u \colon \S^2\to \S^3$ be a $C^1$ minimizing $p$-harmonic tangent map for some $1.3 < p < 3$, and let $\gamma\in (-\frac{1}{2},0]$ if $p\in(1.3,2)$ and
$\gamma\in (-1,0]$ if $p\in[2,3)$.
Then the following inequality holds:
 \begin{equation}\label{eq:Bochnerwithgamma}
 \begin{split}
   \int_{\S^2} |\nabla u|^{p-2+2\gamma}&\bigg\{|\nabla^2 u|^2+(p-2+2\gamma)|\nabla |\nabla u||^2+2\gamma(p-2)\abs{\Delta_\infty u}^2 \bigg\}  \\
  &\le \frac12 \int_{\S^2} |\nabla u|^{p+2+2\gamma} - \int_{\S^2}|\nabla u|^{p+2\gamma},
 \end{split}
\end{equation}
where we adopt the convention that negative powers of $|\nabla u|$ are zero if $|\nabla u|=0$.
\end{theorem}

\begin{proof}
First, consider an arbitrary map $u \colon \S^2 \to \S^3$ of class $C^3$; an approximation argument will later get rid of this assumption. For such maps, we have the Bochner--Weitzenb\"ock formula (see, e.g., \cite[Lemma A.2]{GMM} for a modern proof)
\begin{equation}
\label{eq:Bochner-on-spheres}
    \frac 12 \Delta |\nabla u|^2 
    \geq |\nabla^2 u|^2 + \langle \nabla \Delta u, \nabla u \rangle + |\nabla u|^2 - \frac 12 |\nabla u|^4.
\end{equation}
Fixing some $\delta > 0$, we introduce the quantity 
\[
g\coloneqq|\nabla u|_\delta^{p-4+2\gamma}|\nabla u|^2,
\quad \text{where} \quad
|\nabla u|_\delta\coloneqq(|\nabla u|^2+\delta^2)^{\frac{1}{2}},
\]
and note that $g$ is well-defined and of class $C^2$ thanks to this regularization procedure. Multiply our inequality by $g$ and integrating by parts, we obtain
\begin{equation}\label{eq:Bochneraftertesting}
    \begin{split}
    -\frac{1}{2}\int_{\S^2}\langle \nabla|\nabla u|^2,\nabla g\rangle \geq &\int_{\S^2}g\left(|\nabla^2 u|^2+|\nabla u|^2-\frac{1}{2}|\nabla u|^4\right)
    \\
    &-\int_{\S^2}\langle \Delta u, \divg\{ g\nabla u\}\rangle.
    \end{split}
\end{equation}
Note that this inequality does not involve any third order derivatives of $u$, and second order derivatives are squared. Thus, the inequality can be obtained for any $u \in C^1 \cap W^{2,2}$ by approximation. In consequence of \Cref{second derivatives}, from now on we will work with a $p$-harmonic map $u$ as in the statement of the theorem.
Since
\begin{equation*}
\begin{split}
\nabla g & =  (p-4+2\gamma) \ |\nabla u|_{\delta}^{p-6+2\gamma} \ |\nabla u|^3 \ \nabla|\nabla u| \\
& \pheq + 2 \ |\nabla u|_\delta^{p-4+2\gamma} \ |\nabla u|  \ \nabla |\nabla u|,
\end{split}
\end{equation*}
the term on the left-hand side of \eqref{eq:Bochneraftertesting} can be written as 
\begin{equation}\label{BochnerLHS}
    \begin{split}
        -\frac{1}{2}\int_{\S^2}\langle \nabla|\nabla u|^2,\nabla g\rangle&=-\frac{1}{2}\int_{\S^2}\langle 2|\nabla u|\ \nabla|\nabla u|,\nabla g\rangle
        \\
        &=-\int_{\S^2}|\nabla u|_\delta^{p-4+2\gamma}|\nabla u|^2 \bigg\{(p-4+2\gamma)\frac{|\nabla u|^2}{|\nabla u|_\delta^2}+2\bigg\}|\nabla|\nabla u||^2.
    \end{split}
\end{equation}
For the right-hand side of \eqref{eq:Bochneraftertesting}, we expand $\divg\{ g\nabla u\}$ to obtain 
\begin{equation}
    \begin{split}
        &-\int_{\S^2}\langle \Delta u, \divg\{ g\nabla u\}\rangle
        \\[2ex]
        &=-\int_{\S^2}\left\langle \Delta u, g\Delta u+\bigg( (p-4+2\gamma)|\nabla u|_\delta^{p-6+2\gamma}|\nabla u|^4+2|\nabla u|_\delta^{p-4+2\gamma}|\nabla u|^2\bigg) \Delta_\infty u\right\rangle.
    \end{split}
\end{equation}

Now we can transform this last term by exploiting the (intrinsic) $p$-harmonic map equation
\begin{equation}\label{pequation}
-\divg(|\nabla u|^{p-2} \nabla u)=0.
\end{equation}
On the set $\{\nabla u\neq 0\}$, we can write \eqref{pequation} in non-divergence form as 
\begin{equation}\label{perpeq}
\Delta u+(p-2)\Delta_\infty u=0,
\end{equation}
where the $\infty$-Laplacian $\Delta_\infty u$ was defined in \eqref{eq:infinity-laplacian}.
Substituting $g=|\nabla u|_\delta^{p-4+2\gamma}|\nabla u|^2$ and $\Delta u = (2-p) \Delta_\infty u$ yields
\begin{equation}\label{BochnerRHS}
    \begin{split}
        &-\int_{\S^2}\langle \Delta u, \divg\{ g\nabla u\}\rangle
        \\[2ex]
        &=\int_{\S^2}(p-2)\bigg\{4-p+(p-4+2\gamma)\frac{|\nabla u|^2}{|\nabla u|_{\delta}^2}\bigg\}\abs{\nabla u}_\delta^{p-4+2\gamma}|\nabla u|^2\abs{\Delta_\infty u}^2,
    \end{split}
\end{equation}
Combining \eqref{eq:Bochneraftertesting}, \eqref{BochnerLHS}, \eqref{BochnerRHS}, and rearranging the terms yields 
\begin{equation}\label{DeltaBochner}
    \begin{split}
        \int_{\S^2}|\nabla u|^2&|\nabla u|_\delta^{p-4+2\gamma}\bigg(|\nabla^2 u|^2+\bigg\{(p-4+2\gamma)\frac{|\nabla u|^2}{|\nabla u|_\delta^2}+2\bigg\}|\nabla|\nabla u||^2
        \\[2ex]
        &+(p-2)\bigg\{4-p+(p-4+2\gamma)\frac{|\nabla u|^2}{|\nabla u|_{\delta}^2}\bigg\}|\Delta_\infty u|^2\bigg)
        \\[2ex]
        &\le \frac12 \int_{\S^2} |\nabla u|^6|\nabla u|_\delta^{p-4+2\gamma} - \int_{\S^2}|\nabla u|^{4}|\nabla u|_\delta^{p-4+2\gamma}.
    \end{split}
\end{equation}
Since $p+2\gamma>0$, the right-hand side of \eqref{DeltaBochner} converges to the right-hand side of \eqref{eq:Bochnerwithgamma} as $\delta\to 0$.
For the left-hand side, let $t=\frac{|\nabla u|^2}{|\nabla u|_\delta^2}\in[0,1]$. Then
\begin{equation*}
\begin{split}
    &|\nabla^2 u|^2+\bigg\{(p-4+2\gamma)\frac{|\nabla u|^2}{|\nabla u|_\delta^2}+2\bigg\}|\nabla|\nabla u||^2+(p-2)\bigg\{4-p+(p-4+2\gamma)\frac{|\nabla u|^2}{|\nabla u|_{\delta}^2}\bigg\}|\Delta_\infty u|^2
    \\[3ex]
    &=|\nabla^2 u|^2+\left((p-4+2\gamma)t+2\right)|\nabla|\nabla u||^2+(p-2)\left(4-p+(p-4+2\gamma)t\right)|\Delta_\infty u|^2
    \\[3ex]
    &=|\nabla^2 u|^2+2|\nabla|\nabla u||^2+(p-2)(4-p)|\Delta_\infty u|^2+t(p-4+2\gamma)(|\nabla|\nabla u||^2+(p-2)|\Delta_\infty u|^2)
\end{split}
\end{equation*}
Since $1.3<p<3$, $\gamma\leq 0$, and $|\Delta_\infty u|\leq |\nabla|\nabla u||$, the term involving $t$ is always nonpositive,
\[
t(p-4+2\gamma)(|\nabla|\nabla u||^2+(p-2)|\Delta_\infty u|^2)\leq 0.
\]
Therefore we can estimate it from below by taking $t=1$, which implies that
\begin{align*}
        &|\nabla^2 u|^2+\bigg\{(p-4+2\gamma)\frac{|\nabla u|^2}{|\nabla u|_\delta^2}+2\bigg\}|\nabla|\nabla u||^2+(p-2)\bigg\{4-p+(p-4+2\gamma)\frac{|\nabla u|^2}{|\nabla u|_{\delta}^2}\bigg\}|\Delta_\infty u|^2
        \\
        &\geq |\nabla^2 u|^2+\left( p-2+2\gamma\right) |\nabla|\nabla u||^2+2\gamma (p-2)|\Delta_\infty u|^2.
\end{align*}
By Lemma \ref{le:mixedKCSwithgamma} (Mixed Kato--Cauchy-Schwarz Inequality), we can conclude that the integrand in 
\begin{equation*}
    \begin{split}
        \int_{\S^2}|\nabla u|^2&|\nabla u|_\delta^{p-4+2\gamma}\bigg(|\nabla^2 u|^2+\bigg\{(p-4+2\gamma)\frac{|\nabla u|^2}{|\nabla u|_\delta^2}+2\bigg\}|\nabla|\nabla u||^2
        \\
        &+(p-2)\bigg\{4-p+(p-4+2\gamma)\frac{|\nabla u|^2}{|\nabla u|_{\delta}^2}\bigg\}|\Delta_\infty u|^2\bigg)
    \end{split}
\end{equation*}
is nonnegative. Therefore, the claim \eqref{eq:Bochnerwithgamma} follows by taking the limit $\delta\to 0$ in \eqref{DeltaBochner} and using Fatou's lemma.
\end{proof}

\section{Stability Inequality}

Finally, we turn to the proof of \eqref{stab}.
Note that the argument relies on Theorem \ref{thm:Bochnerwithgamma} and Lemma \ref{le:mixedKCSwithgamma}.

\begin{theorem}[Improved Stability Inequality]\label{thm:stabilitywithgamma}
Let $u \colon \S^2\to \S^3$ be a $C^1$ minimizing $p$-harmonic tangent map for some $1.3<p<3$, and let $\gamma\in (-\frac{1}{2},0]$ if $p\in(1.3,2)$ and
$\gamma\in (-1,0]$ if $p\in[2,3)$.
Then the following inequality holds:
\begin{equation}\label{eq:stabilitywithgamma}
\begin{split}
 (1+\gamma)^2&\int_{\S^2} |\nabla u|^{p-2+2\gamma}\big\{ 3|\nabla|\nabla u||^2+(p-2)\abs{\Delta_\infty u}^2\big\}\\
 & \quad\ge (3-p)\int_{\S^2} |\nabla u|^{p+2+2\gamma}-\frac 34 (3-p)^2\int_{\S^2}|\nabla u|^{p+2\gamma},
\end{split}
\end{equation}
where we adopt the convention that negative powers of $|\nabla u|$ are zero if $|\nabla u|=0$.
\end{theorem}

\begin{proof}
The cases $p\geq 2$ and $p<2$ are almost identical, except that the latter presents some additional technical difficulties.
The case $p\geq 2$ was treated in \cite[Lemma 4.3]{GMM}, although with a slightly smaller range of $\gamma$.
However, as noted in the discussion after inequality \eqref{deltastability1} below, we only need parts of \cite[Lemma 4.3]{GMM} and
\cite[Lemma 4.4(a)]{GMM} where $\gamma$
plays no role. For the first part of this proof we focus on the case $1.3<p<2$.

Hence, let $1.3<p<2$ and let $u\colon\S^{2}\to\S^3$ be a minimizing $p$-harmonic tangent map that is of class $C^1$.
We use the same letter to denote the $0$-homogeneous extension of $u$ to $\R^3$.
By \Cref{second derivatives} \ref{it:Gasteletal}, $u\in W^{2,p}_\textnormal{loc}(\R^3\setminus\{0\})$.
Let $\alpha\in \R^{4}$ be a fixed unit vector and 
\[
\varphi\in W^{1,p}(\R^3)\cap L^\infty(\R^3)\quad \text{with}\quad \supp \varphi\subset \R^3\setminus\{0\}.
\]
For $t\in\R$, let
\[
u_t\coloneqq \frac{u+t \varphi \alpha}{\abs{u+t \varphi \alpha}}.
\]
We will calculate the second variation of the $p$-energy of the perturbation family $u_t$.
However, we need to be careful in the case $p<2$, so we begin with a few estimates. 

The derivative of the map $F(\xi)=\xi/\abs{\xi}$ is 
\[
DF(\xi)=\frac{1}{\abs{\xi}}\left(I-\frac{\xi\otimes \xi}{\abs{\xi}^2}\right),
\]
so we have 
\begin{align*}
    \nabla u_t= \frac{1}{\abs{u+t\varphi\alpha}}(I-u_t\otimes u_t)(\nabla u+t\alpha\otimes\nabla\varphi).
\end{align*}
For convenience, denote 
\begin{align*}
    w_t&\coloneqq u+t\varphi\alpha,
    \\
    A_t&\coloneqq \nabla u+t\alpha\otimes\nabla\varphi,
    \\
    P_t&\coloneqq I-u_t\otimes u_t,
    \\
    b_t&\coloneqq \abs{w_t}^{-1}.
\end{align*}
Then we can write 
\[
\nabla u_t= b_t P_t A_t.
\]
For almost every fixed $x\in \R^3$ and sufficiently small $t$ (so that $\abs{u+t\varphi\alpha}\neq 0$), the map $t\mapsto u_t(x)$ is smooth.
Therefore we can compute the time derivatives of $u_t$:
\begin{align*}
    \partial_t u_t&=\partial_t \frac{w_t}{\abs{w_t}}
    \\
    &= \partial_t \left(F(w_t)\right)
    \\
    &=DF(w_t)\cdot\partial_tw_t
    \\
    &=\frac{1}{\abs{u+t\varphi\alpha}}\left(I-u_t\otimes u_t\right)\cdot (\varphi \alpha)
    \\
    &= b_t P_t\cdot (\varphi \alpha)
\end{align*}
and
\begin{align*}
    \partial_t^2 u_t &=\partial_t \big[b_t P_t\cdot (\varphi \alpha) \big]
    \\
    &= (\partial_t b_t) P_t\cdot(\varphi\alpha)+b_t(\partial_t P_t)\cdot(\varphi\alpha),
\end{align*}
and the time derivatives of $\nabla u_t$:
\begin{align*}
    \partial_t \nabla u_t &= \partial_t \left( b_t P_t A_t\right)
    \\
    &= \left(\partial_t b_t\right)P_t A_t+b_t\left(\partial_t P_t\right) A_t+b_t P_t \left(\partial_t A_t\right)
\end{align*}
and
\begin{align*}
    \partial^2_t \nabla u_t &=\partial_t \bigg\{\left(\partial_t b_t\right)P_t A_t+b_t\left(\partial_t P_t\right) A_t+b_t P_t \left(\partial_t A_t\right)\bigg\}
    \\
    &=\left(\partial^2_t b_t\right)P_t A_t+b_t\left(\partial^2_t P_t\right) A_t+b_t P_t \left(\partial^2_t A_t\right)
    \\
    &\quad +2 \bigg\{\left(\partial_t b_t\right)(\partial_tP_t) A_t+b_t\left(\partial_t P_t\right) (\partial_tA_t)+(\partial_tb_t) P_t \left(\partial_t A_t\right)\bigg\}.
\end{align*}
Then we compute the time derivatives of $b_t$: 
\begin{align*}
    b_t&= \abs{u+t\varphi\alpha}^{-1}
    \\
    \partial_t b_t&= \langle -\frac{u+t\varphi\alpha}{\abs{u+t\varphi\alpha}^3},\varphi\alpha\rangle=-b_t^2\varphi \langle u_t,\alpha\rangle
    \\
    \partial_t^2 b_t&=2b_t^3 \varphi^2 \langle u_t,\alpha\rangle^2-b_t^2\varphi \langle \partial_t u_t,\alpha\rangle,
\end{align*}
of $P_t$:
\begin{align*}
    P_t&= I-u_t\otimes u_t
    \\
    \partial_t P_t&= -\big[(\partial_t u_t)\otimes u_t+u_t\otimes(\partial_t u_t)\big]
    \\
    \partial_t^2 P_t&=-\big[(\partial^2_t u_t)\otimes u_t+2 \ (\partial_t u_t)\otimes (\partial_t u_t)+u_t\otimes (\partial^2_t u_t)\big],
\end{align*}
and of $A_t$:
\begin{align*}
    A_t&= \nabla u+t \alpha\otimes\nabla\varphi
    \\
    \partial_t A_t&= \alpha \otimes \nabla\varphi
    \\
    \partial_t^2 A_t&=0.
\end{align*}
From the above we can derive the estimates 
\begin{align}
    \abs{\nabla u_t}&\lesssim \abs{\nabla \varphi}+\abs{\nabla u},\label{estimate nabla u_t}
    \\
    \abs{\partial_t \nabla u_t}&\lesssim \abs{\varphi}\abs{\nabla\varphi}+\abs{\varphi}\abs{\nabla u}+\abs{\nabla \varphi},\label{estimate d/dt nabla u_t}
    \\
    \abs{\partial_t^2 \nabla u_t}&\lesssim\abs{\varphi}^2\abs{\nabla\varphi}+\abs{\varphi}^2\abs{\nabla u}+\abs{\varphi}\abs{\nabla \varphi}.\label{estimate d^2/dt^2 nabla u_t}
\end{align}
In particular, on the set $\{\nabla u_t\neq 0\}$ we have
\begin{align*}
    \partial_t \abs{\nabla u_t}^p &= \partial_t(\langle \nabla u_t,\nabla u_t\rangle)^{\frac{p}{2}}
    \\
    &=\frac{p}{2} (\langle \nabla u_t,\nabla u_t\rangle)^{\frac{p-2}{2}}\cdot 2 \langle \nabla u_t,\partial_t\nabla u_t\rangle
    \\
    &=p \abs{\nabla u_t}^{p-2} \langle \nabla u_t,\partial_t\nabla u_t\rangle,
\end{align*}
which can be estimated by
\begin{align*}
    \abs{ \partial_t \abs{\nabla u_t}^p}&\lesssim \abs{\nabla u_t}^{p-1}\abs{\partial_t\nabla u_t}
    \\
    &\lesssim \abs{\nabla\varphi}\abs{\nabla u}^{p-1}+\abs{\varphi}\abs{\nabla u}^p+\abs{\varphi}\abs{\nabla\varphi}{\abs{\nabla u}}^{p-1}
    \\
    &\quad+\abs{\nabla\varphi}^p+\abs{\varphi}\abs{\nabla\varphi}^{p-1}\abs{\nabla u}+\abs{\varphi}\abs{\nabla\varphi}^p.
\end{align*}
This estimate also holds on $\{\nabla u_t=0\}$.
Since $\varphi$ is bounded and $\nabla u,\nabla\varphi$ are in $L^p$, $\abs{ \partial_t \abs{\nabla u_t}^p}$ is estimated by an integrable function, uniformly for small $|t|$.
By the dominated convergence theorem, the function 
\[
t\mapsto \int_{\R^3} \abs{\nabla u_t}^p
\]
is differentiable on a neighborhood of $0\in \R$ and we have
\[
\frac{d}{dt}\int_{\R^3} \abs{\nabla u_t}^p =\int_{\R^3} \partial_t\abs{\nabla u_t}^p.
\]
Similarly we compute (for $\nabla u_t\neq 0$)
\begin{align*}
    \partial_t^2 \abs{\nabla u_t}^p &=p(p-2)\abs{\nabla u_t}^{p-4}\langle \nabla u_t,\partial_t\nabla u_t\rangle^2
    \\
    &\quad + p \abs{\nabla u_t}^{p-2}\left(\abs{\partial_t \nabla u_t}^2+\langle \nabla u_t,\partial_t^2\nabla u_t\rangle\right).
\end{align*}
Using the estimates \eqref{estimate nabla u_t},\eqref{estimate d/dt nabla u_t}, and \eqref{estimate d^2/dt^2 nabla u_t},
we obtain
\begin{align*}
    \abs{\partial^2_t \abs{\nabla u_t}^p}&\lesssim \abs{\nabla\varphi}^2\abs{\nabla u}^{p-2}+\abs{\varphi}^2\abs{\nabla\varphi}^p+\abs{\varphi}^2\abs{\nabla u}^p+\abs{\varphi}^2\abs{\nabla\varphi}\abs{\nabla u}
    \\
    &\quad + \abs{\varphi}\abs{\nabla\varphi}\abs{\nabla u}^{p-1}+\abs{\varphi}^2\abs{\nabla\varphi}^{p-1}\abs{\nabla u}+\abs{\varphi}\abs{\nabla\varphi}^p.
\end{align*}
All the terms are integrable, possibly except
\[
\abs{\nabla\varphi}^2\abs{\nabla u}^{p-2}.
\]
Since $p-2<0$, we have no information on the integrability of $|\nabla u|^{p-2}$.
Therefore, we choose a specific test function $\varphi$.
Let $\gamma\in(-\frac{1}{2},0]$ and choose $\varphi$ of the form 
\[
\varphi(r\theta)=f(r)g(\theta),  \quad(r,\theta)\in [0,\infty)\times \S^{2},
\]
where $f\in C^\infty_c(\R_+)$ and 
\[
g\coloneqq\abs{\nabla u}_\delta^\gamma \abs{\nabla u}
\]
for $\delta>0$, where $\abs{\nabla u}_{\delta}\coloneqq(\delta^2+\abs{\nabla u}^2)^{\frac{1}{2}}$. Since $u$ is $C^1$, $\varphi$ is bounded. Furthermore,
\begin{align*}
    \abs{\nabla\varphi}^2&=(f')^2 g^2+r^{-2}f^2\abs{\nabla g}^2
\end{align*}
and 
\[
\abs{\nabla g}= \left(\abs{\nabla u}^\gamma_\delta+\gamma \abs{\nabla u}_\delta^{\gamma-2}\abs{\nabla u}^2\right)\nabla\abs{\nabla u}.
\]
Therefore 
\[
\abs{\nabla \varphi}\lesssim \abs{g}+\abs{\nabla g}
\]
and 
\[
\abs{\nabla g}\lesssim\abs{\nabla^2 u},
\]
where the constants may blow-up as $\delta\to 0$.
By \Cref{second derivatives} \ref{it:Gasteletal},
\[
\nabla u \in W^{1,p}_{\textnormal{loc}}(\R^3\setminus\{0\}),
\]
and therefore our chosen function $\varphi$ is indeed in $W^{1,p}(\R^3)$.
We have 
\begin{align*}
    \abs{\nabla \varphi}^2\abs{\nabla u}^{p-2}&\lesssim \left(\abs{g}^2+\abs{\nabla g}^2\right)\abs{\nabla u}^{p-2}
    \\
    &\lesssim \left(\abs{\nabla u}^2+\abs{\nabla^2 u}^2\right)\abs{\nabla u}^{p-2}
    \\
    &=\abs{\nabla u}^p +\abs{\nabla u}^{p-2}\abs{\nabla^2 u}^2.
\end{align*}
Once again by \Cref{second derivatives} \ref{it:Gasteletal}, $\abs{\nabla u}^{p-2}\abs{\nabla^2 u}^2$ is integrable.
Thus, for this particular $\varphi$, the function 
\[
t\mapsto \int_{\R^3} \abs{\nabla u_t}^p
\]
is twice differentiable on a neighborhood of $0\in \R$, and we have 
\[
\frac{d^2}{dt^2}\bigg|_{t=0} \int_{\R^3} \abs{\nabla u_t}^p = \int_{\R^3} \partial^2_t\big|_{t=0}\abs{\nabla u_t}^p.
\]
Computing the second derivative, we obtain
\begin{equation}\label{presum alpha}
\begin{split}
        0 &\le \frac{1}{p}\frac{d^2}{dt^2}\bigg|_{t=0} \int_{\R^3} \abs{\nabla u_t}^p
        \\
        &=  \int_{\R^3} \bigg\{ |\nabla u|^{p-2} |\nabla \varphi|^2 (1 - |\langle u, \alpha \rangle|^2) \\
& \qquad+  |\nabla u|^{p-2} \varphi^2 (p |\nabla \langle u, \alpha \rangle|^2 + (|\langle u, \alpha \rangle|^2-1) |\nabla u|^2) \\
& \qquad+ (p-2)  |\nabla u|^{p-4} |\langle \nabla \langle u, \alpha \rangle, \nabla \varphi \rangle|^2\bigg\}.
\end{split}
\end{equation}
Summing \eqref{presum alpha} over $\alpha$ in a fixed orthonormal basis of $\R^4$ yields
\begin{equation*}\label{postsum alpha}
0 \le
3 \int_{\R^3}\bigg\{ |\nabla u|^{p-2} |\nabla \varphi|^2
- (3-p)  |\nabla u|^{p} \varphi^2
+ (p-2)  |\nabla u|^{p-4} |\langle \nabla u, \nabla \varphi \rangle|^2\bigg\}.
\end{equation*}
Now we can argue as in \cite[Lemma 1.3.]{SU3} (see \cite[Lemma A.1]{GMM} for a similar proof).By switching to polar coordinates,
factoring out the radial integrals, and optimizing over all $f$, we obtain 
\begin{equation}\label{eq:stabilitywithg}
\begin{split}
&3 \int_{\S^{2}} |\nabla u(\theta)|^{p-2} |\nabla g(\theta)|^2 \ d \theta
+ (p-2) \int_{\S^{2}} |\nabla u(\theta)|^{p-4} |\langle \nabla u(\theta), \nabla g(\theta) \rangle|^2 \ d \theta \\
&\quad - (3-p) \int_{\S^{2}} |\nabla u(\theta)|^p g(\theta)^2 \ d \theta
\ge -\frac{(3-p)^2}{4}
3 \int_{\S^{2}} |\nabla u(\theta)|^{p-2} g(\theta)^2 \ d \theta.
\end{split}
\end{equation}
Our particular choice of $g$ leads to 
\begin{equation}\label{deltastability1}
 \begin{split}
  &3\int_{\S^2} |\nabla u|^{p-2}|\nabla |\nabla u||^2\big\{|\nabla u|^\gamma_\delta + \gamma|\nabla u|^2|\nabla u|^{\gamma -2}_\delta\big\}^2\\
  &\quad + (p-2)\int_{\S^2} |\nabla u|^{p-4} |\langle \nabla u, \nabla |\nabla u|\rangle|^2 \big\{|\nabla u|^\gamma_\delta
 + \gamma|\nabla u|^2|\nabla u|^{\gamma -2}_\delta\big\}^2\\
  &\qquad \geq (3-p)\int_{\S^2} |\nabla u|^{p+2}|\nabla u|^{2\gamma}_\delta - \frac34(3-p)^2 \int_{\S^2} |\nabla u|^p|\nabla u|^{2\gamma}_\delta.
 \end{split}
\end{equation}
The inequality \eqref{deltastability1} holds also for $p\in[2,3)$. For the proof see \cite[Lemma 4.4(a)]{GMM}. In the statement of \cite[Lemma 4.4(a)]{GMM}
there is a slightly different range of $\gamma$, but the same proof works for $-1<\gamma\leq 0$ without any changes (it is only in the later parts of \cite{GMM} where the range of
$\gamma$ plays a role).

For the rest of this proof, we consider $p\in(1.3,3)$. Then \eqref{deltastability1} holds.
Rearranging the left-hand side of \eqref{deltastability1}, we obtain 
\begin{equation}\label{deltastability2}
 \begin{split}
  &\int_{\S^2} |\nabla u|^{p-2}|\nabla u|_\delta^{2\gamma}\left(1+\frac{|\nabla u|^2}{|\nabla u|_\delta^2}\gamma\right)^2 \bigg\{3|\nabla|\nabla u||^2+(p-2)\abs{\Delta_\infty u}^2\bigg\}
  \\
  &\qquad \geq (3-p)\int_{\S^2} |\nabla u|^{p+2}|\nabla u|^{2\gamma}_\delta-\frac34(3-p)^2 \int_{\S^2} |\nabla u|^p|\nabla u|^{2\gamma}_\delta.
 \end{split}
\end{equation}
Since $p+2\gamma>0$, the dominated convergence theorem implies that
\[
\lim_{\delta\to 0 }\int_{\S^2} |\nabla u|^{p+2}|\nabla u|^{2\gamma}_\delta=\int_{\S^2} |\nabla u|^{p+2+2\gamma}
\]
and
\[
\lim_{\delta\to 0 }\int_{\S^2} |\nabla u|^{p}|\nabla u|^{2\gamma}_\delta=\int_{\S^2} |\nabla u|^{p+2\gamma}.
\]
Thus the right-hand side of \eqref{deltastability2} converges to the right-hand side of \eqref{eq:stabilitywithgamma}.
On the other hand, by Lemma \ref{le:mixedKCSwithgamma} (Mixed Kato--Cauchy-Schwarz Inequality), 
\begin{equation*}
    \begin{split}
        &|\nabla u|^{p-2}|\nabla u|_\delta^{2\gamma}\left(1+\frac{|\nabla u|^2}{|\nabla u|_\delta^2}\gamma\right)^2 \bigg\{3|\nabla|\nabla u||^2+(p-2)\abs{\Delta_\infty u}^2\bigg\}
        \\
        &\le C_p|\nabla u|^{p-2+2\gamma}(1+|\gamma|)^2\bigg\{|\nabla^2 u|^2+(p-2+2\gamma)|\nabla |\nabla u||^2+2\gamma(p-2)\abs{\Delta_\infty u}^2 \bigg\}.
    \end{split}
\end{equation*}
By Theorem \ref{thm:Bochnerwithgamma} (Improved Bochner Inequality), the integral of the latter function is bounded from above by
\[
C_p(1+|\gamma|)^2\left( \frac12 \int_{\S^2} |\nabla u|^{p+2+2\gamma} - \int_{\S^2}|\nabla u|^{p+2\gamma}\right),
\]
and hence is finite. Thus we can use the dominated convergence theorem to to conclude the proof by letting $\delta\to 0$ in
\eqref{deltastability2} (note that pointwise convergence of the integrands is evident).
\end{proof}

\section{Regularity --- proof of the main theorem}

Let $u\colon \S^2\to \S^3$ be a minimizing $p$-harmonic tangent map. If $p\in[2,3)$, then, by Theorem \ref{hl thm},
the Hausdorff dimension of the singular set of the $0$-homogeneous extension $U$ of $u$ is at most $n-[p]-1=3-2-1=0$.
It follows that $u$ is regular, since otherwise $U$ would have a singular line, and hence its singular set would
be at least of dimension $1$. Therefore, $u$ is of class $C^1$. In the case $p\in (1,2)$, the same conclusion holds.

\begin{proposition}\label{C1}
    Let $u\colon \S^2\to\S^3$ be a minimizing $p$-harmonic tangent map, where $p>1$. Then $u$ is of class $C^1$.
\end{proposition}
\begin{proof}
    As we noted in the discussion above, the case $p\geq 2$ follows directly from Theorem \ref{hl thm}. Hence,
    suppose that $p\in(1,2)$ and let $u\colon \S^2\to \S^3$ be a minimizing $p$-harmonic tangent map.
    Denote the $0$-homogeneous extension of $u$ by $U$, i.e.,
    \[
    U(x)\coloneqq u\left(\frac{x}{|x|}\right)\quad\text{for $x\neq 0$.}
    \]
    It suffices to show that the singular set of $U$ has Hausdorff dimension $0$. Then $u$ is $C^1$, since
    otherwise singular points of $u$ lead to singular lines of $U$.

    Let $\ell$ be the largest integer such that any minimizing $p$-harmonic tangent map from the unit ball in $\R^j$
    into $\S^3$ is a constant map for each $j=1,\dots,\ell$. Then, by Theorem \ref{hl thm}, $\ell\geq [p]=1$ and
    the singular set of $U$ has Hausdorff dimension at most $3-\ell-1=2-\ell $. Thus, it suffices to show that $\ell\geq 2$,
    i.e., that every minimizing $p$-harmonic tangent map from $\R^2$ into $\S^3$ is constant.

    Suppose for the sake of contradiction that there exists a nonconstant minimizing $p$-harmonic tangent map $\psi\colon \R^2\to \S^3$.
    By Theorem \ref{hl thm}, $\psi$ is $C^1$ on $\R^2\setminus\{0\}$. By $0$-homogeneity,
    \[
    \psi(re^{i\theta})=\gamma(\theta),\quad (r,\theta)\in \R\times (0,2\pi)
    \]
    for some $C^1$ map $\gamma\colon \S^1\to \S^3$, where we identify $\S^1\cong \R/(2\pi \mathbb{Z})$. We have the
    $p$-harmonic map equation
    \[
    -\left(|\gamma'|^{p-2}\gamma'\right)'=|\gamma'|^p \gamma.
    \]
    Taking the inner product with $\gamma'$ and using $\gamma'\perp \gamma$ yields
    \[
    \left((p-2)|\gamma'|^{p-4}(\gamma''\cdot \gamma')\gamma'+|\gamma'|^{p-2}\gamma''\right)\cdot \gamma'=0,
    \]
    which can be simplified to
    \[
    |\gamma'|^{p-2} \gamma''\cdot \gamma'=0.
    \]
    Since $\psi$ is nonconstant, the open set $\{\gamma'\neq 0\}$ is nonempty. On this set, we have
    \[
    \gamma''\cdot \gamma'=0.
    \]
    This means that $|\gamma'|$ is either zero, or equal to some constant $k>0$. By continuity, $|\gamma'|=k$ on all $\S^1$.
    Substituting $|\gamma'|=k$ into the $p$-harmonic map equation yields
    \[
    \gamma''+k^2\gamma=0.
    \]
    Thus, after rotating, we can assume that $\gamma$ is of the form
    \[
    \gamma(\theta)=(\cos k\theta,\sin k\theta,0,0),\quad \theta\in[0,2\pi].
    \]
    Furthermore, $2\pi$-periodicity forces that $k\in \mathbb{Z}$. On the other hand, we have the stability inequality
    \cite[Lemma 3.4.(a)]{GMM}
    \begin{equation}\label{stability S1}
        \begin{split}
            \int_{\S^1} |\gamma'|^{p-4}&\left(3 |\gamma'|^2 ||\gamma '|'|^2 +(p-2)|\langle \gamma',|\gamma'|'\rangle|^2\right)
            \\
            &\geq (3-p)\int_{\S^1} |\gamma'|^{p+2}- \frac{3(2-p)^2}{4}\int_{\S^1} |\gamma'|^p,
        \end{split}
    \end{equation}
    which implies
    \begin{equation}
        (3-p) k^{p+2}\leq \frac{3(2-p)^2}{4} k^p.
    \end{equation}
    Thus, since $p\in(1,2)$,
    \[
    k^2\leq \frac{3(2-p)^2}{4(3-p)}\leq \frac{3}{4}\cdot \frac{1}{1}=\frac{3}{4}<1.
    \]
    However, $k\in \mathbb{Z}$ and $k>0$, which is a contradiction.
\end{proof}

Now that we know that minimizing $p$-harmonic tangent maps from $\S^2$ to $\S^3$ are of class $C^1$, we can apply Theorems \ref{thm:stabilitywithgamma},
\ref{thm:Bochnerwithgamma}, and \ref{le:mixedKCSwithgamma}. Thus, as outlined in Section \ref{outline}, we can prove Theorem \ref{regularity intro}.

\begin{proof}[Proof of Theorem \ref{regularity intro}]
By Theorem \ref{hl thm}, it suffices to show that every minimizing $p$-harmonic tangent map
$u\colon \S^2\to \S^3$ is constant for $p\in[p_0,3)$. By Theorem \ref{C1}, $u$ is $C^1$.

Let $p\in(1.3,3)$ and consider $\gamma\in(-1,0]$ if $p\in[2,3)$ or $\gamma\in(-\frac{1}{2},0]$ if $p\in(1.3,2)$ to be fixed later.
Then, by the stability inequality \eqref{eq:stabilitywithgamma},
\begin{equation*}
\begin{split}
 &3\int_{\S^2} |\nabla u|^{p-2+2\gamma}|\nabla|\nabla u||^2+(p-2)\int_{\S^2}\abs{\nabla u}^{p-2+2\gamma}\abs{\Delta_\infty u}^2\\
 & \quad\ge \frac{3-p}{(1+\gamma)^2}\int_{\S^2} |\nabla u|^{p+2+2\gamma}-\frac{3(3-p)^2}{4(1+\gamma)^2}\int_{\S^2}|\nabla u|^{p+2\gamma}.
\end{split}
\end{equation*}
Combining this with the Kato--Cauchy--Schwarz inequality \eqref{eq:mixedKCSwithgamma} and the Bochner inequality \eqref{eq:Bochnerwithgamma},
we obtain
\begin{align*}
    \frac{3-p}{(1+\gamma)^2}&\int_{\S^2}\abs{\nabla u}^{p+2+2\gamma}-\frac{3(3-p)^2}{4(1+\gamma)^2}\int_{\S^2}\abs{\nabla u}^{p+2\gamma}
    \\
    &\leq \frac{C_p}{2}\int_{\S^2}\abs{\nabla u}^{p+2+2\gamma}+C_p\int_{\S^2}\abs{\nabla u}^{p+2\gamma},
\end{align*}
where
\[
 C_p = \begin{dcases}
 \frac{p+1}{(p + 2\gamma)(p-1)}\quad&\text{if $1.3<p<2$,}
 \\
\frac{3}{p+2\gamma}&\text{if $2\leq p<3$.}
\end{dcases}
 \]
Therefore, we get 
 \begin{equation}\label{regularityIneq}
     \left(\frac{3-p}{(1+\gamma)^2}-\frac{C_p}{2}\right)\int_{\S^2}\abs{\nabla u}^{p+2+2\gamma}+\left(C_p-\frac{3(3-p)^2}{4(1+\gamma)^2}\right)\int_{\S^2}\abs{\nabla u}^{p+2\gamma}\leq 0.
 \end{equation}
Our goal is to find an admissible $\gamma$ for which both of the coefficients 
\[
A(p,\gamma)\coloneqq \frac{3-p}{(1+\gamma)^2}-\frac{C_p}{2}
\]
and 
\[
B(p,\gamma)\coloneqq C_p-\frac{3(3-p)^2}{4(1+\gamma)^2}
\]
are nonnegative, and at least one of them is positive. If we find such $\gamma$, then \eqref{regularityIneq} implies that $u$ is constant.

Suppose first that $p\in(2,3)$. After substituting $\frac{3}{p+2\gamma}$ for $C_p$ and rearranging, the conditions $A(p,\gamma)\geq 0$ and $B(p,\gamma)>0$
can be transformed into 
\begin{equation}\label{p>2:Fcondition}
    F(p,\gamma)\coloneqq 2(p+2\gamma)(3-p)-3(1+\gamma)^2\geq  0
\end{equation}
and 
\begin{equation}\label{p>2:Gcondition}
    G(p,\gamma) \coloneqq 4(1+\gamma)^2-(3-p)^2(p+2\gamma) > 0.
\end{equation}

We will show that, for every $p\in(2,3)$, there is $\gamma\in(-1,0]$ such that \eqref{p>2:Fcondition} and \eqref{p>2:Gcondition} hold. Indeed,
the parameter $\gamma \coloneqq 2-p\in (-1,0]$ is admissible and satisfies 
\begin{align*}
F(p,2-p) & = (p-1)(3-p) > 0, \\
G(p,2-p) & = p(3-p)^2 > 0,
\end{align*}
which concludes the proof. The fact that $\gamma=2-p$ works can also be read off
directly from the diagram below, which illustrates the regions $\{A(p,\gamma)>0\}$ and $\{B(p,\gamma)>0\}$ in the $(p,\gamma)$-plane.
\begin{center}
\begin{tikzpicture}[scale=5, line cap=round, line join=round]
\definecolor{colF}{RGB}{60,150,70}      
\definecolor{colG}{RGB}{70,110,190}     

\def\pathF{--(2.3686,-0.0026)--(2.3793,-0.0133)--(2.3900,-0.0241)--(2.4007,-0.0350)--(2.4114,-0.0460)--(2.4221,-0.0570)--(2.4328,-0.0681)--(2.4435,-0.0793)--(2.4542,-0.0905)--(2.4649,-0.1019)--(2.4756,-0.1133)--(2.4863,-0.1248)--(2.4970,-0.1364)--(2.5077,-0.1480)--(2.5184,-0.1598)--(2.5291,-0.1716)--(2.5398,-0.1836)--(2.5505,-0.1956)--(2.5612,-0.2077)--(2.5719,-0.2200)--(2.5826,-0.2323)--(2.5933,-0.2447)--(2.6040,-0.2573)--(2.6147,-0.2700)--(2.6254,-0.2828)--(2.6361,-0.2957)--(2.6468,-0.3087)--(2.6575,-0.3219)--(2.6682,-0.3353)--(2.6789,-0.3487)--(2.6896,-0.3623)--(2.7003,-0.3761)--(2.7110,-0.3901)--(2.7217,-0.4042)--(2.7324,-0.4185)--(2.7431,-0.4330)--(2.7538,-0.4477)--(2.7645,-0.4626)--(2.7752,-0.4778)--(2.7859,-0.4932)--(2.7966,-0.5089)--(2.8073,-0.5248)--(2.8180,-0.5411)--(2.8287,-0.5577)--(2.8394,-0.5746)--(2.8501,-0.5919)--(2.8608,-0.6097)--(2.8715,-0.6280)--(2.8822,-0.6468)--(2.8929,-0.6662)--(2.9036,-0.6863)--(2.9143,-0.7073)--(2.9250,-0.7292)--(2.9357,-0.7523)--(2.9464,-0.7769)--(2.9571,-0.8035)--(2.9678,-0.8328)--(2.9785,-0.8664)--(2.9892,-0.9081)--(2.9999,-0.9918)}

\def\pathG{--(2.0001,-0.5001)--(2.0170,-0.5085)--(2.0340,-0.5170)--(2.0509,-0.5255)--(2.0679,-0.5339)--(2.0848,-0.5424)--(2.1018,-0.5509)--(2.1187,-0.5594)--(2.1357,-0.5678)--(2.1526,-0.5763)--(2.1696,-0.5848)--(2.1865,-0.5933)--(2.2034,-0.6017)--(2.2204,-0.6102)--(2.2373,-0.6187)--(2.2543,-0.6271)--(2.2712,-0.6356)--(2.2882,-0.6441)--(2.3051,-0.6526)--(2.3221,-0.6610)--(2.3390,-0.6695)--(2.3560,-0.6780)--(2.3729,-0.6865)--(2.3899,-0.6949)--(2.4068,-0.7034)--(2.4237,-0.7119)--(2.4407,-0.7203)--(2.4576,-0.7288)--(2.4746,-0.7373)--(2.4915,-0.7458)--(2.5085,-0.7542)--(2.5254,-0.7627)--(2.5424,-0.7712)--(2.5593,-0.7797)--(2.5763,-0.7881)--(2.5932,-0.7966)--(2.6101,-0.8051)--(2.6271,-0.8135)--(2.6440,-0.8220)--(2.6610,-0.8305)--(2.6779,-0.8390)--(2.6949,-0.8474)--(2.7118,-0.8559)--(2.7288,-0.8644)--(2.7457,-0.8729)--(2.7627,-0.8813)--(2.7796,-0.8898)--(2.7966,-0.8983)--(2.8135,-0.9067)--(2.8304,-0.9152)--(2.8474,-0.9237)--(2.8643,-0.9322)--(2.8813,-0.9406)--(2.8982,-0.9491)--(2.9152,-0.9576)--(2.9321,-0.9661)--(2.9491,-0.9745)--(2.9660,-0.9830)--(2.9830,-0.9915)--(2.9999,-0.9999)}

\fill[colF, opacity=0.4]
  (2,-1) -- (2,0) -- (2.3686,0) \pathF -- (3,-1) -- (2,-1) -- cycle;

\fill[colG, opacity=0.4]
  (3,-1) -- (3,0) -- (2,0) -- (2,-0.5) \pathG -- cycle;

\draw[colF, thick, dashed] (2.3686,0) \pathF;
\draw[colG, thick, dashed] (2.0001,-0.5001) \pathG;

\draw[thick] (2,-1) rectangle (3,0);

\draw (2,-1) -- (2,-1.02) node[below] {$2$};
\draw (2.5,-1) -- (2.5,-1.02) node[below] {$2.5$};
\draw (3,-1) -- (3,-1.02) node[below] {$3$};

\draw (2,-1) -- (1.98,-1) node[left] {$-1$};
\draw (2,-0.5) -- (1.98,-0.5) node[left] {$-0.5$};
\draw (2,0) -- (1.98,0) node[left] {$0$};

\draw[->] (3,-1) -- (3.15,-1) node[right] {$p$};
\draw[->] (2,0) -- (2,0.15) node[above] {$\gamma$};

\node[font=\large] at (2.25,-0.85) {$A>0$};
\node[font=\large] at (2.78,-0.18) {$B>0$};

\end{tikzpicture}
\end{center}

Now, suppose that $p\in(1.3,2)$.
Then 
\[
A(p,\gamma)=\frac{3-p}{(1+\gamma)^2}-\frac{p+1}{2(p-1)(p+2\gamma)}.
\]
Since $p+2\gamma>0$, the condition $A(p,\gamma)\geq 0$ is equivalent to
\begin{equation}\label{p<2:Acondition}
    2(3-p)(p-1)(p+2\gamma)-(1+\gamma)^2(p+1)\geq 0,
\end{equation}
which can be written as 
\[
-\left(p+1\right)\gamma^2+\left(-4p^2+14p-14\right)\gamma+\left(-2p^3+8p^2-7p-1\right)\geq 0.
\]
For fixed $p$, this is a quadratic inequality for $\gamma$, which has a solution if and only if its discriminant $\Delta$ is nonnegative. This discriminant is exactly 
\begin{gather*}
\Delta = 8\left(p^4-11p^3+39p^2-53p+24\right) = 8(p-1)(p-p_0)(p-3)(p-p_1), \\
\text{where } p_0 \coloneqq\tfrac{7 - \sqrt{17}}{2} \approx 1.438, \ p_1 \coloneqq\tfrac{7 + \sqrt{17}}{2} \approx 5.562.
\end{gather*}
For $1.3<p<2$, the inequality \eqref{p<2:Acondition} has a solution $\gamma$ if and only if $p\in[p_0,2]$.
It remains to show that for all such $p$ we can find a solution $\gamma$ in the interval $(-\frac{1}{2},0]$.
In fact, there is a single $\gamma$ that works for all $p\in[p_0,2)$, namely
the unique solution $\gamma_0$ of 
\[
    2(3-p_0)(p_0-1)(p_0+2\gamma_0)-(1+\gamma_0)^2(p_0+1)=0,
\] 
which is equivalent to $A(p_0,\gamma_0)=0$.
We have
\[
\gamma_0 = \frac{-2p_0^2+7p_0-7}{p_0+1} = 1-p_0 \approx -0.438,
\]
and so $\gamma_0\in(-\frac{1}{2},0]$. Furthermore,
for all $p\in[p_0,2)$ we have 
\begin{align*}
    \frac{\partial}{\partial p}&\left(2(3-p)(p-1)(p+2\gamma_0)-(1+\gamma_0)^2(p+1)\right)
    \\
    &=-\gamma_0^2+(-8p+14)\gamma_0-6p^2+16p-7
    \\
    &=-6p^2+\left(16-8\gamma_0\right)p+\left(-\gamma_0^2+14\gamma_0-7\right)
    \\
    &\geq0,
\end{align*}
and so 
\begin{align*}
        2(3-p)&(p-1)(p+2\gamma_0)-(1+\gamma_0)^2(p+1)
        \\
        &\geq 2(3-p_0)(p_0-1)(p_0+2\gamma_0)-(1+\gamma_0)^2(p_0+1)
        \\
        &= 0.
\end{align*}
Hence, with this particular choice of $\gamma=\gamma_0$, the condition \eqref{p<2:Acondition} holds for all $p\in[p_0,2)$.
Furthermore, for $p\in [p_0,2)$ we also have
\begin{align*}
    B(p,\gamma_0)=\frac{p+1}{(p+2\gamma_0)(p-1)}-\frac{3(3-p)^2}{4(1+\gamma_0)^2}>0,
\end{align*}
\begin{figure}[htbp]
    \centering
    \includegraphics[width=0.5\textwidth]{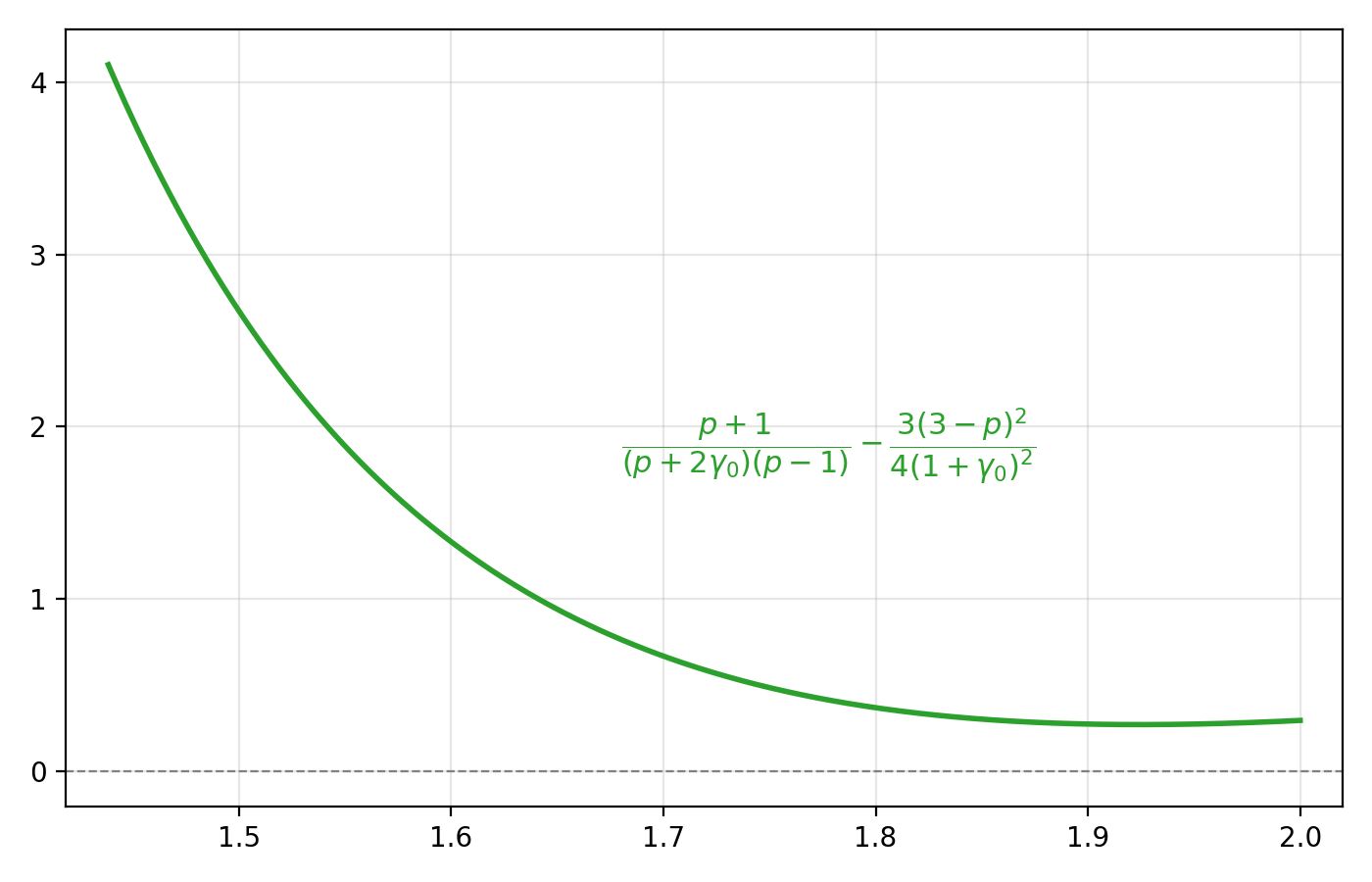}
    \caption{The graph of $B(p,\gamma_0)$ for $p\in[p_0,2)$}\label{fig:B_plot}
\end{figure}
which can be seen from Figure \ref{fig:B_plot} below.

Therefore, for all $p\geq p_0$, minimizing $p$-harmonic maps $u\colon B^3\to\S^3$ are regular.
\end{proof}

\bibliographystyle{abbrv}%
\bibliography{references}%
\end{document}